\documentclass[11 pt]{amsart}
\usepackage{amsmath}
\usepackage{latexsym,amsfonts,amssymb,mathrsfs}
\usepackage{mathdots}
\usepackage{color}
\usepackage [all,2cell,color]{xy}
\usepackage{enumitem} 
\usepackage{ dsfont } 

\usepackage[T1]{fontenc}
\usepackage{url}

\usepackage{tikz-cd}
 \usepackage[utf8]{inputenc}
 \DeclareFontFamily{U}{min}{}
 \DeclareFontShape{U}{min}{m}{n}{<-> udmj30}{}

\UseAllTwocells
\usepackage[
textwidth=3cm,
textsize=small,
colorinlistoftodos]
{todonotes}

\usepackage{tikz-cd}
\usetikzlibrary{positioning} 
\usetikzlibrary{decorations.pathreplacing}
\usetikzlibrary{arrows, decorations.markings} 
\pgfarrowsdeclarecombine{twotriang}{twotriang}{stealth}{stealth}%
{stealth}{stealth}
\tikzset{arrow/.style={-stealth}}
\tikzset{arrowshorter/.style={-stealth, shorten <=2pt, shorten >=2pt}}
\tikzset{arrowmuchshorter/.style={-stealth, shorten <=7pt, shorten >=6pt}}
\tikzset{mono/.style={>-stealth}} 
\tikzset{epi/.style={-twotriang}} 
\tikzset{twoarrowlonger/.style={double,double distance=1.5pt,
shorten <=5pt,shorten >=6pt,
decoration={markings,mark=at position -4pt with {\arrow[scale=1.75]{>}}},
preaction={decorate}}} 

\usetikzlibrary{backgrounds,shapes,arrows,calc} 

\pgfdeclarelayer{background}
\pgfsetlayers{background,main}

\usepackage{expl3} 
\ExplSyntaxOn  
\def\minusone#1%
{\int_eval:n {#1 - 1}}
\ExplSyntaxOff
\ExplSyntaxOn   
\def\sumplusone#1#2%
{\int_eval:n {#1 +#2+1}}
\ExplSyntaxOff

\usepackage{hyperref}

\tikzset{mapstikz/.style={-stealth, 
decoration={markings,mark=at position 0pt with {\arrow[scale=0.5]{|}}}, preaction={decorate}}}

\tikzset{
    dot/.style={circle,draw,fill,inner sep=1pt}
}

\theoremstyle{plain}   
\newtheorem{thm}{Theorem}[section] 
\makeatletter\let\c@thm\c@thm\makeatother
\newtheorem{cor}{Corollary}[section]
\makeatletter\let\c@cor\c@thm\makeatother
\newtheorem{lem}{Lemma}[section]
\makeatletter\let\c@lem\c@thm\makeatother
\newtheorem{prop}{Proposition}[section]
\makeatletter\let\c@prop\c@thm\makeatother

\makeatletter\let\c@claim\c@thm\makeatother

\makeatletter\let\c@question\c@thm\makeatother

\newtheorem*{unnumberedtheoremA}{Theorem A}
\newtheorem*{unnumberedtheoremB}{Theorem B}

\theoremstyle{definition}

\newtheorem{defn}{Definition}[section]
\makeatletter\let\c@defn\c@thm\makeatother
\newtheorem{const}{Construction}[section]
\makeatletter\let\c@const\c@thm\makeatother
\newtheorem{notn}{Notation}[section]
\makeatletter\let\c@notn\c@thm\makeatother

\theoremstyle{remark}

\newtheorem{rmk}{Remark}[section]
\makeatletter\let\c@rmk\c@thm\makeatother
\newtheorem{ex}{Example}[section]
\makeatletter\let\c@ex\c@thm\makeatother

\makeatletter\let\c@observationn\c@thm\makeatother
\newtheorem{digression}{Digression}[section]
\makeatletter\let\c@digression\c@thm\makeatother

\makeatletter
\let\c@equation\c@thm
\numberwithin{equation}{section}
\makeatother

\usepackage[capitalise]{cleveref}

\newcommand{\newrefformat}[2]{}

\crefname{lem}{Lemma}{Lemmas}
\crefname{thm}{Theorem}{Theorems}
\crefname{defn}{Definition}{Definitions}
\crefname{notn}{Notation}{Notations}
\crefname{const}{Construction}{Constructions}
\crefname{prop}{Proposition}{Propositions}
\crefname{rmk}{Remark}{Remarks}
\crefname{cor}{Corollary}{Corollaries}
\crefname{equation}{Display}{Displays}
\crefname{ex}{Example}{Examples}

\newcommand{\cC}{\mathcal{C}}
\newcommand{\cD}{\mathcal{D}}

\newcommand{\cM}{\mathcal{M}}
\newcommand{\cN}{\mathcal{N}}

\newcommand{\cS}{\mathcal{S}}

\newcommand{\cat}{\cC\!\mathit{at}}
\newcommand{\CAT}{\cC\!\mathit{AT}}

\newcommand{\set}{\cS\!\mathit{et}}
\newcommand{\SET}{\cS\!\mathit{ET}}

\newcommand{\sset}{\mathit{s}\set}
\newcommand{\SSET}{\mathit{s}\SET}

\newcommand{\qcat}{\mathit{q}\cat}
\newcommand{\pder}{\mathit{p}\cD\!\mathit{er}}
\newcommand{\pDER}{\mathit{p}\cD\!\mathit{ER}}
\newcommand{\catfin}{cat}

\DeclareMathOperator{\Ob}{Ob}
\DeclareMathOperator{\colim}{colim}
\DeclareMathOperator{\coeq}{coeq}
\DeclareMathOperator{\id}{id}
\DeclareMathOperator{\ho}{ho}
\DeclareMathOperator{\Mor}{Mor}

\DeclareMathOperator{\Hom}{Hom}

\DeclareMathOperator{\Map}{Map} 

\DeclareMathOperator{\op}{op}

\newcommand{\ssset}{\mathit{s}\mathit{s}\set}

\newcommand{\rep}[1]{\mathbb Ho(#1)}

\usepackage{pigpen}
\newcommand{\po}{\ar@{}[dr]|{\text{\pigpenfont R}}}
\newcommand{\pb}{\ar@{}[dr]|{\text{\pigpenfont J}}}

\begin{document}

\title{A model structure on prederivators for $(\infty,1)$-categories}

\author[D.~Fuentes-Keuthan]{Daniel Fuentes-Keuthan}
\address{Department of Mathematics,
Johns Hopkins University, Baltimore
}
\email{danielfk@jhu.edu}

\author[M.~K\k{e}dziorek]{Magdalena K\k{e}dziorek}
\address{Mathematical Institute, Utrecht University}
\email{m.kedziorek@uu.nl}

\author[M.~Rovelli]{Martina Rovelli}
\address{Department of Mathematics,
Johns Hopkins University, Baltimore
}
\email{mrovelli@math.jhu.edu}

\thanks{
The second author was supported by the NWO Veni grant 639.031.757. 
The third author was partially funded by the Swiss National Science Foundation, grant P2ELP2\textunderscore172086.
}

\maketitle

\begin{abstract}
By theorems of Carlson and Renaudin, the theory of $(\infty,1)$-categories embeds in that of prederivators. 
The purpose of this paper is to give a two-fold answer to the inverse problem: understanding which prederivators model $(\infty,1)$-categories, either strictly or in a homotopical sense. First, we characterize which prederivators arise on the nose as prederivators associated to quasicategories. Next, we put a model structure on the category of prederivators and strict natural transformations, and prove a Quillen equivalence with the Joyal model structure for quasicategories.
\end{abstract}

\setcounter{tocdepth}{1}

\tableofcontents

\section*{Introduction}
The notion of prederivator appeared independently and with different flavours in works by Grothendieck \cite{grothendieck}, Heller \cite{heller} and Franke \cite{franke}. A \emph{prederivator} is a contravariant $2$-functor $\mathbb D\colon\cat^{\op}\to\CAT$, which we regard as minimally recording the homotopical information of a given $(\infty,1)$-category.

The idea is that the value $\mathbb D(J)$ at a small category $J$ represents the homotopy category of $J$-shaped diagrams of the desired homotopy theory. For example: the prederivator $D_{\cC}$ associated to any ordinary $1$-category $\cC$ records the functor categories
$D_{\cC}(J):=\cC^J$, the prederivator $\rep{\cM}$ associated to any model category $\cM$, defined by $\rep{\cM}(J):=\cM^J[W^{-1}]$, is obtained by inverting the class $W$ of levelwise weak equivalences of $J$-shaped diagrams in $\cM$, and the prederivator $\rep{X}$ associated to any quasi-category $X$ can be realized as $\rep{X}(J):=\ho(X^J)$, where $\ho\colon\sset\to\cat$ denotes the left adjoint to the nerve functor.

The prederivator associated to an $(\infty,1)$-category of some flavour is suitably homotopy invariant. The initial intuition might suggest that passing to prederivators would cause a significant loss of information, given that it involves taking homotopy categories. However, as pointed out in Shulman's note \cite{shulmanblog}, ``derivators seem to suffice for all sorts of things that one might want to do in an $(\infty,1)$-category''.

A heuristic explanation can be attributed to the fact that the data of the prederivator associated to $X$  records the collection $\ho(X^J)$, where the parameter $J$ runs over  $\cat^{\op}$, and is therefore a true enhancement of the bare homotopy category $\ho(X)$ of the $(\infty,1)$-category $X$.
Additional evidence is given in the work of Riehl-Verity \cite{RV4,RVbook}, who show that much of $(\infty,1)$-category theory can be recovered by truncating each functor $(\infty,1)$-category $X^Y$ between two $(\infty,1)$-categories $X$ and $Y$ to their homotopy category $\ho(X^Y)$.

A more rigorous validation of the fact that prederivators carry all the relevant information of a given $(\infty,1)$-category was given by Renaudin in \cite{renaudin}. He proves that the bi-localization of the $2$-category of left-proper combinatorial model structures at the class of Quillen equivalences embeds in the $2$-category of prederivators and pseudonatural transformations.
Carlson shows in \cite{carlson} an analogous result in a different framework. He proved that the functor $\mathbb Ho$ actually gives a simplicial embedding $\mathbb Ho\colon\qcat\to\pder^{st}$ of the category of (small) quasicategories in that of (small) prederivators and strict natural transformations.

Inspired by the classical theorem of Brown representability, Carlson \cite{carlson}
raised the question of whether the essential image of $\mathbb Ho$ could be characterized.
In this paper, we provide such a characterization.
More precisely, we recognise that all prederivators of the form $\rep{X}$ meet certain conditions, which we introduce in \cref{sectionquasirepresentable} under the terminology of ``quasi-representability'' (\cref{definitionrepresentable}). Essentially, a prederivator is \emph{quasi-representable} if and only if its value at the level of objects commutes with certain colimits, its value at the level of morphisms is suitably determined by that at the level of objects, and its underlying simplicial set is a quasi-category.

One of the main results of this paper is the following theorem, which appears in the paper as \cref{characterizationrepresentable} and describes which prederivators arise from quasi-categories in a strict sense.

\begin{unnumberedtheoremA}
A prederivator $\mathbb D$ is quasi-representable if and only if it is of the form $\mathbb D\cong \rep{X}$ for some quasi-category $X$.
\end{unnumberedtheoremA}

Next we concentrate on the homotopical analysis, identifying a suitable notion of weak equivalence of prederivators so that the homotopy category of $\pder^{st}$ is equivalent to that of $\qcat$.
We show that this class of weak equivalences is part of a model structure on $\pder^{st}$ in \cref{transferredmodelstructure} and we prove that this model structure is equivalent to the model structure for quasi-categories in \cref{Quillenequivalence}.
The following is the main result of the paper, and it in particular validates prederivators as a model of $(\infty,1)$-categories.

\begin{unnumberedtheoremB}
There exists a cofibrantly generated model structure on the category $\pder^{st}$ of (small) prederivators and strict natural transformations that is Quillen equivalent to the Joyal model structure on the category $\sset$ of simplicial sets.
\end{unnumberedtheoremB}

The desired model structure is transferred from the Joyal model structure on $\sset$ using a certain  functor $R\colon\pder^{st}\to\sset$, which will be defined in Construction 1.13 and was already used to prove \cite[Proposition 2.9]{carlson}. It is interesting to observe that, despite the functor $\mathbb Ho\colon\sset\to\pder^{st}$
having been more extensively studied in the history of prederivators, it cannot be used to transfer a model structure given that it does not admit an adjoint on either side.

Finally, we conclude the paper with an explicit description of the generating cofibrations, and a characterization of the fibrant objects and the acyclic fibrations in terms of suitable lifting properties. In particular, every quasi-representable prederivator is fibrant, and any fibrant prederivator is weakly equivalent to a quasi-representable one.

\textbf{Outline of the paper.} In Section 1 we introduce the category $\pder^{st}$ of small prederivators and strict natural transformations, and the further structure that $\pder^{st}$ possesses, which will be used later in the paper. In Section 2 we identify the image of $\mathbb Ho\colon\qcat\to\pder^{st}$ as the class of quasi-representable prederivators. In Section $3$ we transfer the Joyal model structure from $\sset$ to $\pder^{st}$, study its properties, and prove the desired Quillen equivalence.

\textbf{Further directions.}
The Quillen equivalence from Theorem B justifies that prederivators have the same \emph{homotopy theory} as  quasi-categories. In future work we aim to produce a rigorous comparison of the \emph{category theory} of prederivators and that of quasicategories\footnote{Roughly speaking, the \emph{homotopy theory} of a model of $(\infty,1)$-categories is comprised of the information stored in the homotopy category of that model. It detects, for instance, when two given $(\infty,1)$-categories are equivalent to one another. On the other hand, the \emph{category theory} of a model of $(\infty,1)$-categories focuses on the homotopy $2$-category of the model. This focuses, for instance, on whether there is an adjunction between two given $(\infty,1)$-categories, whether an $(\infty,1)$-category is stable, or whether an element of an $(\infty,1)$-category $Q$ is the limit of a given diagram in $Q$.
}.

The standard method to access the category theory of a model of $(\infty, 1)$-categories presented by a model category is to upgrade the model category to an $\infty$-cosmos, in the sense of Riehl-Verity.  This is done by providing a model categorical enrichment over the Joyal model structure on simplicial sets.
The model structure that we construct on the category of prederivators is unfortunately not enriched over the Joyal model structure, so we cannot conclude easily that the category of fibrant prederivators forms an $\infty$-cosmos.

However, given that Carlson proves that $\mathbb Ho\colon\qcat\to\sset$ is simplicially fully faithful,
the functor $\mathbb Ho$ induces an isomorphism of $\infty$-cosmoi onto its image. We plan to return to this topic in a future project and compare the $2$-category of quasicategories (which has been developed e.g.~in \cite{Joyalnotes,htt,RVbook}) with the $2$-category of (quasi-representable) prederivators (which has been developed e.g.~in \cite{groth,grothmonoidal}).

\textbf{Acknowledgements.}
The authors are grateful to Emily Riehl for carefully reading an earlier draft of this paper and providing valuable feedback. We also would like to thank Kevin Carlson and Viktoriya Ozornova for many interesting discussions on the subject.

\section{The category of prederivators}

A prederivator is typically defined as a $2$-functor
$\cat^{\op}\to\CAT$, where $\cat$ denotes the category of small categories and $\CAT$ denotes the very large category of large categories. Several authors  (see e.g.~\cite{carlson,groth,gps,renaudin}) have considered the (very large) category $\pDER^{st}$ of (large) prederivators and strict natural transformations, and the (very large) category $\pDER^{pseudo}$ of (large) prederivators and pseudonatural transformations. In this paper we are concerned with the former point of view.

We aim to study the homotopy theory of prederivators and compare it to the homotopy theory of quasi-categories, as presented by the Joyal model structure on the category $\sset$ of small simplicial sets. In order to employ standard references for model category theory and enriched category theory, we will focus on a type of prederivators that assemble into a locally small category\footnote{While working with a locally small category of prederivators simplifies the exposition, this restriction is not necessary. An alternative approach would be to define prederivators to be $2$-functors $\cat^{op}\to\CAT$, or $2$-functors $\cat^{\op}\to\cat$, which assemble into a category that is not necessarily locally small, and to compare it with the very large category $\SSET$ of \emph{large} simplicial sets.
To see that all of our constructions go through unchanged upon ascending to a larger universe, see \cite{lowuniverses}.}.
To this end, we replace
the large category $\CAT$ with the locally small category $\cat$ as a target category of prederivators and we follow Grothendieck's approach from \cite[Chapter 1, Section 2]{grothendieck} in replacing the (large) indexing category $\cat^{op}$ with a smaller one:  the category $\catfin$ of ``homotopically finite categories'', which was already considered by Carlson in \cite{carlson}. With these arrangements, we show that the category of $2$-functors $\catfin^{\op}\to\cat$ will turn out to be locally small.

\begin{defn}[{\cite[Definition 0.3]{carlson}}]
\label{catfin}
A category $J$ is \emph{homotopy finite}, also known as \emph{finite direct}, if the nerve $NJ$
has finitely many nondegenerate simplices; equivalently, if $J$ is finite, skeletal, and admits
no nontrivial endomorphisms. We denote by $\catfin$ the full subcategory of $\cat$ consisting of homotopy finite categories.
\end{defn}

\begin{ex}
Any category $[n]$ is homotopically finite, so $\catfin$ includes $\Delta$ fully faithfully.
\end{ex}

\begin{ex}
The category containing a free isomorphism is not homotopically finite.
\end{ex}

\begin{ex}
The category associated to any infinite group is not homotopically finite.
\end{ex}

\begin{prop}
\label{catfinsmall}
The $2$-category $\catfin$ is small.
\end{prop}

\begin{proof}
We show that the class of objects of $\catfin$ is countable, so in particular a set. If we denote by $C_{m,n}$ the class of homotopy finite categories which have precisely $m$ objects and $n$ arrows, then each $C_{m,n}$ is finite, and the set of objects
$$\Ob(\catfin)=\coprod_{m,n}C_{m,n},$$
is therefore countable.

Next, we observe that, for any homotopically finite categories $J$ and $K$, the class $\Hom_{\catfin}(J,K)$ of functors $J\to K$ is a set, given that it is a subset of $\Hom_{\set}(\Mor(J),\Mor(K))$.

Finally, we observe that for any homotopically finite categories $J$ and $K$ and functors $F,G\colon J\to K$, the class of strict natural transformations $F\Rightarrow G$ is a set, given that it is a subset of $\Hom_{\set}(\Ob(J),\Mor(K))$.
\end{proof}

We mimick the usual definition of prederivator in our context.

\begin{defn}
A \emph{(small) prederivator} is a $2$-functor $\catfin^{\op}\to\cat$.
We denote by $\pder^{st}$ the category of (small) prederivators and (strict) natural transformations.
We denote by $\Hom_{\pder^{st}}(\mathbb D,\mathbb E)$ the homset between two prederivators $\mathbb D$ and $\mathbb E$.
\end{defn}

\begin{prop}
The category $\pder^{st}$ is locally small.
\end{prop}

\begin{proof}
Given that $\cat$ is locally small, and that $\catfin$ was shown to be a small $2$-category in \cref{catfinsmall},
the category $\pder^{st}$ is an instance of the functor category as described in \cite[\textsection 2.2]{Kelly}, which is locally small.
\end{proof}

In this paper, all prederivators will be small and all natural transformations will be strict unless specified otherwise.

\begin{rmk}
\label{limitsinpder}
Since $\cat$ is complete and cocomplete, as mentioned in \cite[\textsection3.3]{Kelly} the category $\pder^{st}$ of prederivators is complete and cocomplete, with limits and colimits computed pointwise in $\cat$.
\end{rmk}

The category of $\pder^{st}$ is the underlying category of a $2$-category whose hom-categories $\Map_{\pder^{st}}(\mathbb D,\mathbb E)$ are given by strict natural transformations $\mathbb D\to\mathbb E$ and modifications of such. This is discussed in more detail in \cite[\textsection2.1]{groth}.

As $\pder^{st}$ is a category of enriched presheaves, standard constructions from enriched category theory apply.

\begin{notn}[{\cite[\textsection2.4]{Kelly}}]
\label{yonedaembedding}
Let $D\colon\cat\to\pder^{st}$
denote the Yoneda embedding, with the representable $D_K$ at a category $K$ being the 2-functor with values
$$D_K(J):=K^J.$$
\end{notn}

\begin{rmk}
\label{Dcommuteswithproduct}
The Yoneda embedding $D\colon\cat\to\pder^{st}$ preserves binary products, i.e., there are isomorphisms of prederivators
$$D_{J\times K}\cong D_J\times D_K.$$
\end{rmk}

As proven in {\cite[Section 4]{heller}}, the category $\pder^{st}$ is cartesian closed.

\begin{prop}
\label{pDercartesianclosed}
The category $\pder^{st}$ is cartesian closed, with the internal hom $\mathbb D^{\mathbb E}$ given by
$$\mathbb D^{\mathbb E}(J):=\Map_{\pder^{st}}(D_J\times\mathbb E,\mathbb D).$$
\end{prop}

As proven in \cite{mr}, the above cartesian closure can be used to enrich $\pder^{st}$ over $\sset$.

\begin{prop}
\label{pDersimplicial}
The category $\pder^{st}$ is the underlying category of a simplicial category whose hom-simplicial sets $\Map_{\pder^{st}}(\mathbb D,\mathbb E)_\bullet$
are given by
$$\Map_{\pder^{st}}(\mathbb D,\mathbb E)_n:=\Ob(\mathbb E^{\mathbb D}([n]))=\Hom_{\pder^{st}}(D_{[n]}\times\mathbb D,\mathbb E).$$
\end{prop}

We warn the reader that $\Map_{\pder^{st}}(\mathbb D,\mathbb E)_\bullet$ is used to denote the simplicial enrichment, and $\Map_{\pder^{st}}(\mathbb D,\mathbb E)$ is used to denote the categorical enrichment. The following remark clarifies the relationship between them.

\begin{rmk}
For any prederivators $\mathbb D$ and $\mathbb E$, there is a  canonical map
$$\alpha\colon\Map_{\pder^{st}}(\mathbb D,\mathbb E)_\bullet\to N\Map_{\pder^{st}}(\mathbb D,\mathbb E).$$
The simplicial map $\alpha$ is induced on the set of $n$-simplices by postcomposition with the \emph{underlying diagram} functors
$$dia_J^{[n]}\colon(\mathbb E^{D_{[n]}})(J)\to\mathbb E(J)^{[n]},$$
which are natural in $J$ and assemble into a map of prederivators $\mathbb E^{D_{[n]}}\to\mathbb E(\bullet)^{[n]}$. We refer the reader to \cite{groth} for more details on the underlying diagram functors.

When $\mathbb E=D_K$ is represented by a category, the underlying diagram functors can be checked to be isomorphisms and the two enrichements agree, in the sense that $\alpha$ becomes an isomorphism
$$\alpha\colon\Map_{\pder^{st}}(\mathbb D,D_K)_\bullet\cong N\Map_{\pder^{st}}(\mathbb D,D_K).$$

This is not the case in general, even when $\mathbb E=\rep{X}$ is the prederivator associated to a quasi-category $X$. For instance, at the level of $1$-simplices the underlying diagram functor can be identified with
$$\ho(X^{J\times[1]})\to\ho(X^J)^{[1]},$$
and the corresponding map $\alpha$
is not bijective.
\end{rmk}

The enriched Yoneda Lemma from \cite[\textsection2.4]{Kelly} specializes to the following.

\begin{prop}
\label{yonedalemma}
There is a natural isomorphism of categories
$$\Map_{\pder^{st}}(D_{J},\mathbb E)\cong \mathbb E(J).$$
In particular, the isomorphisms induce natural bijections at the level of objects
$$\Hom_{\pder^{st}}(D_{J},\mathbb E)\cong\Ob(\mathbb E(J)).$$
\end{prop}

Given that $\Delta$ is a small category, that $\sset$ is locally small and that $\pder^{st}$ is cocomplete by \cref{limitsinpder}, \cite[Construction 1.5.1]{RiehlCHT} specializes to the following adjunction.

\begin{const}
\label{mainadjunction}
The restriction $D_{[\bullet]}\colon\Delta\subset\cat\to\pder^{st}$ of the Yoneda embedding is a cosimplicial prederivator, and therefore induces an adjunction
$$L\colon\sset\rightleftarrows\pder^{st}\colon R.$$
The left adjoint $LX$ is the left Kan extension of $D_{[\bullet]}\colon\Delta\subset\cat\to\pder^{st}$ along the Yoneda embedding $\Delta \subset \sset$, explicitly
$$LX=\int^{[n]\in\Delta}\Hom_{\sset}(\Delta[n],X)\cdot D_{[n]}=\colim_{\Delta[n_i]\to X} D_{[n_i]},$$
and the right adjoint $R\mathbb D$, which we call the \emph{underlying simplicial set} of $\mathbb D$, is defined by 
$$(R\mathbb D)_n:=\Hom_{\pder^{st}}(D_{[n]},\mathbb D)\cong\Ob(\mathbb D([n])).$$
\end{const}

We now collect three properties of the functors $R$ and $L$ that will be needed later.

\begin{prop}
\label{adjointsofR}
The functor $R$ admits a right adjoint, and in particular it preserves colimits.
\end{prop}

\begin{proof}
We first observe that the functor $R$ can be expressed as the following composite
$$\pder^{st}\to\cat^{\Delta^{\op}}\to\set^{\Delta^{\op}}=\sset,$$
where the first functor is the restriction along the inclusion of the discrete $2$-category $\Delta^{\op}$ into the full $2$-subcategory $\catfin^{\op}$ of $\cat$, and the second functor is induced by the functor $\Ob\colon\cat\to\set$.
In particular, since we are considering $\Delta^{\op}$ as a discrete $2$-category, the category $\cat^{\Delta^{\op}}$ of ordinary functors coincides with the category of $2$-functors.

Knowing from \cref{catfinsmall} that $\catfin$ is a small $2$-category, we can evoke \cite[Thorem 4.50]{Kelly} to say that the restriction along $\Delta^{\op}\to\catfin^{\op}$ admits a right $1$-categorical adjoint, given by the enriched right Kan extension.
The adjoint pair
$$\Ob\circ-\colon\cat^{\Delta^{\op}}\rightleftarrows\sset\colon codisc\circ-$$
and the adjoint pair
$$(\Delta^{\op}\hookrightarrow\catfin^{\op})^*\colon\pder^{st}\rightleftarrows\cat^{\Delta^{\op}}\colon Ran_{\Delta^{\op}\hookrightarrow\catfin^{\op}}$$
compose to an adjoint pair
$$R\colon\pder^{st}\rightleftarrows\cat^{\Delta^{\op}}\rightleftarrows\sset\colon U,$$
as desired.
\end{proof}

\begin{rmk}
\label{rightinverse3}
For any $K\in\cat$ there is a natural isomorphism of simplicial sets
$$RD_K\cong NK.$$
\end{rmk}

The functor $R$ is also a left inverse for $L$.

\begin{prop}
\label{rightinverse2}
For any simplicial set $X$, the unit of the adjunction from \cref{mainadjunction} gives an isomorphism
$$\eta_X\colon X\cong RL(X).$$
In particular, the functor $L$ is fully faithful and the functor $R$ is a left inverse for $L$.
\end{prop}

\begin{proof}
We first prove that the unit of a representable simplicial set,
\[\eta_{\Delta[n]}\colon \Delta[n]\rightarrow RL(\Delta[n])\] 
is an isomorphism.
The component $m$ of the unit map,
\[{\eta_{\Delta[n]}}_m\colon \Delta[n]_m\rightarrow(RL\Delta[n])_m,\]
can be identified with the canonical isomorphism
\[
\begin{array}{rcl}
\Delta[n]_m&=&\Hom_{\sset}(\Delta[m], \Delta[n])\\
&\cong&\Hom_{\cat}([m],[n])\\
&\cong&\Hom_{\pder^{st}}(D_{[m]}, D_{[n]}))\\
&=&\Hom_{\pder^{st}}(D_{[m]}, L(\Delta[n]))\\
&=&(RL\Delta[n])_m.
\end{array}\]
As a consequence, the unit $\eta_{\Delta[n]}$ is an isomorphism.

We now show that the unit $\eta_X$ is an isomorphism for any simplicial set $X$.
Given the canonical identification
$$\phi\colon \colim_{\Delta[n_i]\to X} \Delta[n_i]\cong X$$
and the fact that both $R$ and $L$ respect colimits, we obtain a further identification
$$\phi'\colon \xymatrix{\colim_{\Delta[n_i]\to X}(RL \Delta[n_i])\cong RL(\colim_{\Delta[n_i]\to X} \Delta[n_i])\ar[r]_-{\cong}^-{RL\phi}&
RLX }.$$
Using the universal property of colimits and the naturality of $\eta$, a straightforward check shows that the following diagram commutes
$$\xymatrix{
\colim_{\Delta[n_i]\to X} \Delta[n_i]\ar[d]^-{\phi}_{\cong}\ar[rr]^-{\colim\eta_{\Delta[n_i]}}&& \colim_{\Delta[n_i]\to X} RL\Delta[n_i],\ar[d]^-{\cong}_-{\phi'}\\
X\ar[rr]_-{\eta_X}&&RLX
}$$
and in particular there is an isomorphism
\[\eta_X \cong \colim_{\Delta[n_i]\to X}\eta_{\Delta[n_i]}.\]
The right hand map is an isomorphism, given that it is a colimit of isomorphisms, and so we conclude that $\eta_X$ is an isomorphism as well.
\end{proof}

Unlike many other left adjoints arising from cosimplicial objects (such as geometric realization or homotopy category of simplicial sets), the functor $L$ does not respect products, as shown by the following example. This obstruction will play an important role in a later discussion on the $(\infty,2)$-categorical nature of $\pder^{st}$. See Digression \ref{finaldigression} for more details.

\begin{ex}
\label{Ldoesnotrespectproducts}
We show that the canonical map
$$L(\Delta[1]\times\Delta[1])\to L(\Delta[1])\times L(\Delta[1])$$
is not an isomorphism of prederivators.

First, by construction of $L$ and by \cref{Dcommuteswithproduct} we observe that
$$L(\Delta[1])\times L(\Delta[1])\cong D_{[1]}\times D_{[1]}\cong D_{[1]\times[1]},$$
while, using the fact that $\Delta[1]\times\Delta[1]\cong\Delta[2]\amalg_{\Delta[1]}\Delta[2]$, we obtain that
$$\begin{array}{rcl}
L(\Delta[1]\times\Delta[1])&\cong&L(\Delta[2]\amalg_{\Delta[1]}\Delta[2])\\
&\cong&L(\Delta[2])\amalg_{L(\Delta[1])}L(\Delta[2])\\
&\cong&D_{[2]}\amalg_{D_{[1]}}D_{[2]}.\\
\end{array}$$
Under these identifications, the comparison map can be written in the form
$$D_{[2]}\amalg_{D_{[1]}}D_{[2]}\to D_{[1]\times[1]}.$$
If we denote by $\Gamma$ the span shape category $\bullet\leftarrow\bullet\rightarrow\bullet$, it is enough to show that
the induced map
$$f\colon \Ob((D_{[2]}\amalg_{D_{[1]}}D_{[2]})(\Gamma))\to\Ob(D_{[1]\times[1]}(\Gamma))$$
which can be rewritten as
$$f\colon\Ob([2]^{\Gamma})\amalg_{\Ob([1]^{\Gamma})}\Ob([2]^{\Gamma})\to\Ob(([1]\times[1])^{\Gamma})$$
is not surjective.
To this end, we observe that the span shape $\Gamma\to[1]\times[1]$, whose image is $(0,1)\leftarrow(0,0)\rightarrow(1,0)$,
defines an object of the category $([1]\times[1])^{\Gamma}$ that is not in the image of $f$. Indeed, by unravelling the definitions, one can see that any diagram $d'\colon\Gamma\to[1]\times[1]$ lying in the image of $f$ has to factor through one of the non degenerate $2$-simplices of $[1]\times[1]$.
\end{ex}

\section{Quasi-representable prederivators}
\label{sectionquasirepresentable}

The Yoneda embedding $D\colon\cat\to\pder^{st}$ from \cref{yonedaembedding} provides a natural way to produce a prederivator from any category. There is in fact a canonical construction, which appears and is an object of study in several sources such as \cite{carlson,groth,GPS2,lenz,RV5}, to extend the Yoneda embedding along the nerve inclusion $N\colon\cat\to\sset$, and produce a prederivator from any quasi-category (and in fact from any simplicial set). This construction makes use of the ``homotopy category'' of a simplicial set.

\begin{rmk}
\label{homotopycategory}
We recall that the nerve functor admits a left adjoint $\ho\colon\sset\to\cat$, which acts as a $1$-truncation, see e.g. \cite[Definition 1.1.10]{RVbook} sending a simplicial set to its \emph{homotopy category}. By \cite[Lemma 1.1.12]{RVbook}, when $X$ is a quasi-category the homotopy category $\ho(X)$ has as its set of objects $\Ob(\ho(X)):=X_0$, and as its set of morphisms the homotopy classes of $1$-simplices of $X$.
As defined in \cite[Definition 1.1.7.]{RVbook}, two $1$-simplices $f$ and $g$ from $x$ to $y$ are homotopic if there exists a $2$-simplex $\sigma$ such that
$$d_0(\sigma)=f,d_1(\sigma)=g\text{ and }d_2(\sigma)=s_0(x).$$
This set of morphisms can be described as the coequalizer
$$\Mor(\ho(X)):=\coeq\left(X_0\times^{d_2}_{X_1}X_2\rightrightarrows X_1\right)$$
of the structure maps induced by the faces $d_0,d_1\colon X_2\to X_1$.
\end{rmk}

\begin{defn}
For any simplicial set $X$, its \emph{homotopy prederivator} is the prederivator $\rep{X}$ that is defined by 
$$\rep{X}(J):=\ho(X^{NJ}).$$
Upon restricting the domain, this construction defines a functor $\mathbb Ho\colon\qcat\to\pder^{st}$.
\end{defn}

\begin{rmk}
\label{DversusHo}
Using the isomorphism $\ho(NJ)\cong J$ for any $J\in\cat$, one can see that there is an isomorphism of prederivators
$$D_J\cong\rep{NJ}.$$
\end{rmk}

In \cite{carlson}, Carlson observed that this functor is simplicial when the category $\pder^{st}$ is endowed with the simplicial structure recalled in \cref{pDersimplicial}, and he proves that it is a simplicial embedding.

\begin{thm}[{\cite[Theorem 2.1]{carlson}}]
\label{theoremcarlson}
The functor $\mathbb Ho\colon\qcat\to\pder^{st}$ is simplicially fully faithful, i.e., it induces isomorphisms of simplicial sets
$$\mathbb Ho\colon\Map_{\qcat}(X,X')_\bullet\cong\Map_{\pder^{st}}(\rep{X},\rep{X'})_\bullet.$$
\end{thm}

The functor $R$ provides a left simplicial inverse for $\mathbb Ho$.

\begin{lem}
\label{rightinverse}
The functor $R$ is a left inverse for $\mathbb Ho$, i.e., for any simplicial set $X$ there is a natural isomorphism of simplicial sets
$$R\rep{X}\cong X.$$
\end{lem}

\begin{proof}
By definition, we obtain the natural identification
$$(R\rep{X})_n:=\Ob(\rep{X}[n])=\Ob(\ho(X^{\Delta[n]}))=(X^{\Delta[n]})_0=X_n,$$
which yields an isomorphism $R\rep{X}\cong X$, as desired.
\end{proof}

Since $L$ allows us to embed simplicial sets into prederivators, we get the following.

\begin{prop}
\label{pDercotensored}
The simplicial category $\pder^{st}$ from \cref{pDersimplicial} is 
cotensored over $\sset$, with cotensors $\mathbb D^Y:=\mathbb D^{LY}$, given componentwise by
$$\mathbb D^Y(J):=\mathbb D^{LY}(J)\cong\Map_{\pder^{st}}(D_J\times LY,\mathbb D).$$
\end{prop}

\begin{proof}
We first note that a direct verification shows that
$$\Map_{\pder^{st}}(\mathbb D,\mathbb E)_\bullet\cong R(\mathbb E^{\mathbb D}).$$
Using the adjunction from \cref{mainadjunction}, the fact that $\pder^{st}$ is cartesian closed from \cref{pDercartesianclosed} and the definition of the candidate cotensor, we then have natural bijections
$$\begin{array}{rclr}
\Hom_{\sset}(X,\Map_{\pder^{st}}(\mathbb D,\mathbb E)_\bullet)&\cong&\Hom_{\sset}(X,R(\mathbb E^{\mathbb D}))\\
&\cong&\Hom_{\pder^{st}}(LX,\mathbb E^{\mathbb D})\\
&\cong&\Hom_{\pder^{st}}(\mathbb D,\mathbb E^{LX})\\
&\cong&\Hom_{\pder^{st}}(\mathbb D,\mathbb E^{X}),\\
\end{array}$$
as desired.
\end{proof}

The following records a certain compatibility of the cotensor construction and the functor $\mathbb Ho$, which will be used later.

\begin{prop}
\label{Drespectscotensors}
The functor $\mathbb Ho$ preserves cotensors with respect to (nerves of) categories, i.e., there are isomorphisms of prederivators
$${\rep{X}}^{NK}\cong\rep{X^{NK}}.$$
\end{prop}

\begin{proof}
By direct inspection and using \cref{Dcommuteswithproduct}, we see that there are isomorphisms of categories
$$\begin{array}{rcl}
{\rep{X}}^{NK}(J)&:=&\Map_{\pder^{st}}(D_J\times D_{K},\rep{X})\\
&\cong&\Map_{\pder^{st}}(D_{J\times K},\rep{X})\\
&\cong&\rep{X}(J\times K)\\
&\cong&\ho(X^{NJ\times NK})\\
&\cong&\ho((X^{NK})^{NJ})\\
&\cong& \rep{X}^{D_{NK}}(J)\\
&\cong& \rep{X}^{NK}(J),
\end{array}$$
as desired.
\end{proof}

We now address the question of identifying the essential image of the functor
$$\mathbb Ho\colon\qcat\subset\sset\to\pder^{st}.$$
It follows from the definition of the functor $\mathbb Ho$ that any prederivator of the form $\mathbb D=\rep{X}$ must send finite coproduts to finite products,
$$\mathbb D(NJ\amalg NK)=\ho(X^{NJ\amalg NK})\cong\ho(X^{NJ})\times\ho(X^{NK})=\mathbb D(J)\times\mathbb D(K),$$
but this condition is clearly not sufficient.

In this section, we will show that the prederivators of the form $\mathbb D=\rep{X}$ for some quasi-category $X$ are precisely the prederivators which satisfy the following three conditions, which we suggestively call the ``quasi-representable prederivators''.

\begin{defn}
\label{definitionrepresentable}
A prederivator $\mathbb D\colon\catfin^{\op}\to\cat$ is \emph{quasi-representable} if the following three conditions hold.
\begin{enumerate}
\item For any category $J\in\catfin$, the counit of the adjunction $(L,R)$
$$LNJ\cong LR\rep{NJ}\stackrel{\epsilon_{NJ}}{\longrightarrow}\rep{NJ}\cong D_{J}$$
 induces a bijection
 $$
 \Ob(\mathbb D(J))\cong
  \Hom_{\pder^{st}}(D_J,\mathbb D)\stackrel{\epsilon_{NJ}^*}{\longrightarrow}
  \Hom_{\pder}(LNJ,\mathbb D)
  \cong
 \Hom_{\sset}(NJ,R\mathbb D).$$
\item For any category $J\in\catfin$ the function induced by the underlying diagram functor
$$dia^{[1]}_J\colon\mathbb D([1]\times J)\to \mathbb D(J)^{[1]}$$
at the level of objects realizes a coequalizer diagram
$$\xymatrix@C=2pc{\Ob\left(\mathbb D([0]\times J)\times^{d_2}_{\mathbb D([1]\times J)}\mathbb D([2]\times J)\right) \ar@<+0.7ex>[r] \ar@<-0.5ex>[r] & \Ob(\mathbb D([1]\times J))\ar[r]^-{\Ob(dia^{[1]}_J)} & \Ob\left(\mathbb D(J)^{[1]}\right)}$$
for the maps induced at the level of objects by
$$\xymatrix@C=3.5pc{\mathbb D([0]\times J)\times^{d_2}_{\mathbb D([1]\times J)}\mathbb D([2]\times J) \ar[r]^-{pr_2} &\mathbb D([2]\times J) \ar@<+0.7ex>[r]^{\mathbb D(d^0\times J)} \ar@<-0.5ex>[r]_{\mathbb D(d^1\times J)}& \mathbb D([1]\times J). }$$

\item The underlying simplicial set $R\mathbb D$ is a quasi-category.
\end{enumerate}
\end{defn}

The conditions in the definition above are motivated by the following result, which is a source of numerous examples.

\begin{prop}
\label{representableisrepresentable}
For any quasi-category $X$, the prederivator $\rep{X}$ is quasi-representable.
\end{prop}

\begin{proof}
For Condition (1) of \cref{definitionrepresentable}, we note that for any $J\in\catfin$ \cref{rightinverse} yields a natural bijection
$$
\Hom_{\sset}(NJ,R\rep{X})\cong\Hom_{\sset}(NJ,X)\cong\Ob(\rep{X}(J)),
$$
as desired.

In order to show that Condition (2) holds for $\rep{X}$ for any simplicial set $X$, we first prove it for $J=[0]$.
In this case the diagram that we ought to show is a coequalizer is
$$\xymatrix@C=1.4pc{\Ob\left(\rep{X}([0])\times^{d_2}_{\rep{X}([1])}\rep{X}([2])\right)\ar@<+0.7ex>[r] \ar@<-0.5ex>[r] & \Ob(\rep{X}([1])) \ar[d]_{\Ob(dia_{[0]}^{[1]})} \\
& \Ob\left(\rep{X}([0])^{[1]}\right).}$$
By the definition of $\rep{X}$, this diagram can be expressed as
$$X_0\times^{d_2}_{X_1}X_2\rightrightarrows X_1\to\Mor(\ho(X)).$$
which is a coequalizer diagram by the description of the homotopy category $\ho(X)$ given in \cref{homotopycategory}.

For the general case, the diagram that we need to show is a coequalizer is
$$\xymatrix@C=1.4pc{\Ob\left(\rep{X}([0]\times J)\times^{d_2}_{\rep{X}([1]\times J)}\rep{X}([2]\times J)\right)\ar@<+0.7ex>[r] \ar@<-0.5ex>[r] & \Ob(\rep{X}([1]\times J))\ar[d]_{\Ob(dia_J^{[1]})} \\ & \Ob\left(\rep{X}(J)^{[1]}\right).}$$
By \cref{Drespectscotensors}, this diagram can be rewritten as
$$\xymatrix@C=1.4pc{\Ob\left(\rep{X^{NJ}}([0])\times^{d_2}_{\rep{X^{NJ}}([1])}\rep{X^{NJ}}([2])\right) \ar@<+0.7ex>[r] \ar@<-0.5ex>[r] & \Ob(\rep{X^{NJ}}([1])) \ar[d]_{\Ob(dia_{[0]}^{[1]})} \\ & \Ob\left(\rep{X^J}([0])^{[1]}\right),}$$
and this was already observed to be a coequalizer because the prederivator $\rep{X^{NJ}}$ satisfies condition (2) for $J=[0]$.

Finally, Condition (3) is a consequence of \cref{rightinverse}, which asserts that $R\rep{X}\cong X$.
\end{proof}

While Condition (3) in the definition above is self-explanatory, we  elaborate on the meaning of conditions (1) and (2).

\begin{rmk}
Given that for every prederivator $\mathbb D$ one finds the identification 
$$\Ob\left(\mathbb D(J)^{[1]}\right)=\Mor(\mathbb D(J)),$$
Condition (2) essentially describes how the value of a quasi-representable derivator $\mathbb D$ on morphisms is completely determined by the $\set$-valued functor $\Ob\circ\mathbb D$.
\end{rmk}

Condition (1), however, seems less transparent. We now explain how it can be interpreted as requiring compatibility of $\mathbb D$ with a certain class of colimits at the level of objects.

Recall that the nerve $N\colon\cat\to\sset$ does not respect colimits in general. Therefore, when taking colimits of (nerves of) categories, we need to be careful. We focus on  diagrams of the following form, for which the issue does not exist.

\begin{defn}
Let $\{[n_i]\}_{i}\colon I\to\Delta\subset\cat$ be a diagram. We say that \emph{the colimit of the diagram $\{[n_i]\}_{i}$ is created in $\sset$} if the colimit in $\sset$ of the diagram $\{\Delta[n_i]\}_{i}\colon I\to\sset$, obtained by postcomposing with the Yoneda embedding, is the nerve of some category $J\in\catfin$, i.e, if there is an isomorphism of simplicial sets
$$\colim^{\sset}_{i\in I}\Delta[n_i]\cong NJ.$$
In this case, by applying $\ho\colon\sset\to\cat$, we in particular get an isomomorphism of categories.
$$\colim^{\cat}_{i\in I}[n_i]\cong J,$$
which justifies the terminology.
\end{defn}

The following remark implies on the one hand that any category $J$ is a colimit of a diagram whose colimit is created in $\sset$, and on the other hand that any quasi-representable derivator is determined on objects by its value on all $[n]$'s.

\begin{rmk}\label{remarkcolimits}
For every category $J$, the nerve $NJ$ is a presheaf and can therefore canonically be written as a colimit of representables $\Delta[n_i]$'s,
$$NJ\cong\colim^{\sset}_{\Delta[n_i]\to NJ}\Delta[n_i],$$
indexed over the diagram $\Delta\downarrow X\to\Delta$ (see e.g. \cite[\textsection3.1]{hovey}). This presentation is natural in $J$.
By definition, the colimit of this diagram is created in $\sset$, and we obtain the isomorphism of categories
$$J\cong\colim^{\cat}_{[n_i]\to J}[n_i].$$
This essentially describes the fact that the category $J$ can be built by taking a copy of $[0]$ for any object of $J$, a copy of $[1]$ for any morphism of $J$, a copy of $[2]$ for any commutative triangle in $J$, a copy of $[3]$ for any triple of composable arrows in $J$, and so on.
\end{rmk}

\begin{prop}
A prederivator $\mathbb D\colon\catfin^{\op}\to\cat$ satisfies Condition (1) of \cref{definitionrepresentable} if and only if it satisfies the following condition.
\begin{enumerate}
\item[(1')] For any diagram of categories $\{[n_i]\}_{i}$ with colimit created in $\sset$, the canonical map
$$\mathbb D(\colim^{\cat} [n_i]))\to\lim{}^{\cat}\mathbb D([n_i])$$
 induces bijections at the level of objects
 $$\Ob(\mathbb D(\colim^{\cat} [n_i])))\cong\Ob(\lim{}^{\cat}\mathbb D([n_i]))\cong \lim{}^{\set}\Ob(\mathbb D([n_i])).$$
\end{enumerate}
\end{prop}
\begin{proof}
Because we can decompose any category $J$ as a colimit created in $\sset$ by \cref{remarkcolimits}, Condition (1) becomes immediately equivalent to the assertion that for any colimit in $\sset$ of the form $\colim \Delta[n_i] \cong NJ$, there is an isomorphism $$\Ob(\mathbb D(\colim^{\cat} [n_i])))\cong\Hom_{\sset}(\colim^{\sset}\Delta[n_i]),R\mathbb D).$$
Now observe that for any diagram $\{n_i\}_i$ whose colimit is created in $\sset$ there are natural isomorphisms
$$\begin{array}{rclr}
\Hom_{\sset}(\colim^{\sset}\Delta[n_i]),R\mathbb D)&\cong&\lim^{\set}\Hom_{\sset}([n_i],R\mathbb D)\\
&\cong&\lim^{\set} (R\mathbb D)_n\\
&\cong&\lim^{\set}\Ob(\mathbb D( [n_i])),&
\end{array}$$
and so if the isomorphism $$\Ob(\mathbb D(\colim^{\cat} [n_i])))\cong\Hom_{\sset}(\colim^{\sset}\Delta[n_i]),R\mathbb D)$$ of Condition (1) holds, then so must the isomorphism $$\Ob(\mathbb D(\colim^{\cat} [n_i])))\cong\lim{}^{\set}\Ob(\mathbb D( [n_i]))$$
of Condition (1'), and vice versa.
\end{proof}

\begin{ex}
For any $C\in\cat$, the representable prederivator $D_C$ (as described in \cref{yonedaembedding}) is quasi-representable.
\end{ex}

\begin{ex}
The following prederivators fail to be quasi-representable.
\begin{enumerate}
\item If $X$ is a non empty quasi-category, the prederivator $\rep{X}\amalg D_{[0]}$ fails to satisfy condition (1) of \cref{definitionrepresentable}. To see this, we observe for instance that this prederivator does not send coproducts to products, even at the level of objects:
$$\begin{array}{rcl}
\Ob((\rep{X}\amalg D_{[0]})([0]\amalg[0]))&=&\Ob((\ho(X^{[0]\amalg[0]})\amalg[0])\\
&\cong&\Ob((\ho(X\times X))\amalg[0])\\
&\cong&\Ob((\ho(X)\times\ho(X))\amalg[0])\\
&\not\cong&\Ob((\ho(X)\amalg [0])\times (\ho(X)\amalg[0]))\\
&=&\Ob((\rep{X}\amalg D_{[0]})([0]))\times\Ob((\rep{X}\amalg D_{[0]})([0])).
\end{array}$$

\item If $K$ is a non discrete category, the functor constant at $K$ fails to satisfy condition (2) of \cref{definitionrepresentable}. To see this, we observe that the diagram
$$\Ob(K)\rightrightarrows \Ob(K)\stackrel{\id}{\longrightarrow}\Mor(K),$$
where the parallel arrows are both identities on $\Ob(K)$,
is not a coequalizer.
\item If $Y$ is simplicial set that is not a quasi-category, the functor $\rep{Y}$ fails to satisfy condition (3) of \cref{definitionrepresentable}, given that
$$R\rep{Y}\cong Y\not\in\qcat$$
as follows from \cref{rightinverse}.
\end{enumerate}
\end{ex}

The terminology ``quasi-representable'' is justified by the following.

\begin{thm}
\label{characterizationrepresentable}
A prederivator $\mathbb D$ is quasi-representable if and only if it lies in the image of $\mathbb Ho\colon\qcat\to\pder^{st}$, i.e., if it is of the form
$$\mathbb D\cong \rep{X}$$
for a quasi-category $X$.
\end{thm}

The proof makes use of a the following result, which shows how the underlying simplicial set of a quasi-representable prederivator uniquely determines the prederivator.
\begin{lem}
\label{testonfirstrow}
The functor $R$ reflects isomorphisms between quasi-repre\-sentable prederivators, i.e.,
given two quasi-representable prederivators $\mathbb D$ and $\mathbb E$,
$$\text{if }R\mathbb D\cong R\mathbb E\text{ then }\mathbb D\cong\mathbb E.$$
\end{lem}

\begin{proof}
As a consequence of Condition (1) of \cref{definitionrepresentable} we obtain that $\mathbb E$ and $\mathbb D$ agree at the level of objects, namely, for any $J\in\catfin$ there are natural bijections
$$\begin{array}{rcl}
\Ob(\mathbb D(J))&\cong&\Hom_{\sset}(NJ,R\mathbb D)\\
&\cong&\Hom_{\sset}(NJ,R\mathbb E)\\
&\cong&\Ob(\mathbb E(J)).
\end{array}$$

As a consequence of Condition (2) of \cref{definitionrepresentable}, we see that $\mathbb E$ and $\mathbb D$ agree at the level of morphisms, namely, for any $J\in\catfin$ there are natural bijections
\[
\begin{tikzcd}
\Mor(\mathbb D(J))\arrow[d, "\cong"]\\
\coeq\left(\Ob(\mathbb D([0]\times J)\times^{d_2}_{\Ob(\mathbb D([1]\times J))}\Ob(\mathbb D([2]\times J)))\rightrightarrows\Ob(\mathbb D([1]\times J))\right)\arrow[u]\arrow[d, "\cong"]\\
\coeq\left(\Ob(\mathbb E([0]\times J)\times^{d_2}_{\Ob(\mathbb E([1]\times J))}\Ob(\mathbb E([2]\times J)))\rightrightarrows\Ob(\mathbb E([1]\times J))\right)\arrow[u]\arrow[d, "\cong"]\\
\Mor(\mathbb E(J)).\arrow[u]
\end{tikzcd}
\]

Finally, a straightforward check shows that the natural bijections above are compatible with identities, source and target maps, and compositions, so that for any $J\in\catfin$ we get a natural isomorphism of categories
$$\mathbb D(J)\cong\mathbb E(J),$$
and therefore an isomorphism of prederivators
$\mathbb D\cong\mathbb E$, as desired.
\end{proof}

\begin{rmk} Notice that $R$ does not reflect isomorphisms between non quasi-representable prederivators. Consider for instance the map
$$L(\Delta[1]\times\Delta[1])\to L(\Delta[1])\times L(\Delta[1])$$
from \cref{Ldoesnotrespectproducts}. It was proven not to be an isomorphism of prederivators, but it is sent by $R$ to an isomorphism of simplicial sets as a consequence of \cref{rightinverse2} and of the fact that $R$ commutes with products.
\end{rmk}

We now finish the proof of \cref{characterizationrepresentable} by showing that any quasi-representable prederivator $\mathbb D$ is in the essential image of $\mathbb Ho\colon\qcat\to\pder^{st}$.

\begin{prop}
\label{inverseforrepresentable}
If a prederivator $\mathbb D$ is quasi-representable, then there is an isomorphism of prederivators
$$\mathbb D\cong\rep{R\mathbb D}.$$
\end{prop}

\begin{proof}
By the definition of a quasi-representable prederivator, $R\mathbb D$ is a quasi-category.
By \cref{rightinverse}, we have an isomorphism of simplicial sets
$$R\rep{R\mathbb D}\cong R\mathbb D,$$
and by \cref{testonfirstrow} we get an isomorphism of prederivators
$$\rep{R\mathbb D}\cong \mathbb D,$$
as desired.
\end{proof}

\section{The model category of prederivators}

In this section, we put a model structure on $\pder^{st}$ by transferring the Joyal model structure using the functor $R\colon\pder^{st}\to\sset$, and we prove that the induced Quillen pair is in fact a Quillen equivalence.

\begin{thm}
\label{transferredmodelstructure}
The category $\pder^{st}$ admits the transferred model structure using the  functor $R\colon\pder^{st}\to\sset$, where by definition fibrations and weak equivalences are created by $R$. Furthermore, with respect to this model category structure, the adjunction $(L,R)$ is a Quillen pair.
\end{thm}

For further reference, we record here the main properties for the Joyal model structure.

Denote by $\mathbb I$ the free living isomorphism category, i.e., the category containing two objects and two inverse isomorphisms between them. 

\begin{thm}[Joyal]
\label{Joyalmodelstructure}
There exists a cofibrantly generated model structure on the category $\sset$ in which
\begin{itemize}
\item the cofibration are precisely the monomorphisms;
\item the weak equivalences are precisely the categorical equivalences;
\item the fibrant objects are precisely the quasi-categories;
\item the fibrations between quasi-categories are precisely the maps between quasi-categories that have the right lifting properties with respect to the inner horn inclusions $\Lambda^k[n]\hookrightarrow\Delta[n]$ for $n>0$ and $0<k<n$ and with respect to either inclusion $\Delta[0]\to N\mathbb I$;
\item the acyclic fibrations are precisely the maps that have the right lifting properties with respect to the boundary inclusions $\partial\Delta[n]\hookrightarrow\Delta[n]$ for $n>0$;
\item the generating cofibrations are the boundary inclusions $\partial\Delta[n]\hookrightarrow\Delta[n]$ for $n>0$.
\end{itemize}
\end{thm}

While this model structure is cofibrantly generated for formal reasons, there is no explicit description of the class of generating acyclic cofibrations.

We will apply the following classical transfer theorem (see e.g.~\cite[Theorem 11.3.2]{hirschhorn}).

\begin{thm}[Kan]
\label{theoremtransferredmodelstructure}
Let $\cM$ be a cofibrantly generated model category with a  set of generating cofibrations $I$ and a set of generating acyclic
cofibrations $J$, let $\cN$ be a complete and cocomplete category, and let
$$F\colon\cM\rightleftarrows\cN\colon G$$
be an adjunction.
Suppose that the following conditions hold.
\begin{enumerate}
\item The left adjoint $F$ preserves small objects; this is the case in particular when the right adjoint preserves filtered colimits.
\item The right adjoint $G$ takes maps that can be constructed
as a transfinite composition of pushouts of elements of $F(J)$ to weak equivalences.
\end{enumerate}
Then, there is a cofibrantly generated model category structure on $\cN$ in which
\begin{itemize}
\item the set $F(I)$ is
a set of generating cofibrations,
\item the set $F(J)$ is a set of generating acyclic cofibrations,
\item the weak equivalences are the maps that $G$ takes to weak equivalences in $\cM$ and
\item the fibrations are the maps that $G$ takes to fibrations in $\cM$.
\end{itemize}
Furthermore, with respect to this model category structure, the adjunction $(F, G)$ is a Quillen pair.
\end{thm}

\begin{proof}[Proof of \cref{transferredmodelstructure}]
The Joyal model structure is cofibrantly generated, as mentioned in \cref{Joyalmodelstructure}, and by \cref{limitsinpder}, the category $\pder^{st}$ is complete and cocomplete. We now check that the conditions of \cref{theoremtransferredmodelstructure} hold for the desired adjunction.
\begin{enumerate}
\item[(1)] The functor $R$ preserves all colimits by \cref{adjointsofR}, so in particular it preserves filtered colimits.
\item[(2)]
Since $R$ preserves all colimits, to check the second condition it is enough to show that the image under $RL$ of any generating trivial cofibration of the Joyal model structure is again a trivial cofibration. The natural isomorphism $RL\cong \id_{\pder^{st}}$ from \cref{rightinverse2} finishes the proof.
\qedhere
\end{enumerate}
\end{proof}

The following validates the model category of prederivators as a model for the homotopy theory of $(\infty,1)$-categories.

\begin{thm}
\label{Quillenequivalence}
The Quillen pair
$$L\colon\sset\rightleftarrows\pder^{st}\colon R$$
from \cref{transferredmodelstructure} is a Quillen equivalence.
\end{thm}

The proof makes use of the following standard result about Quillen equivalences. 

\begin{prop}\label{QuillenequivalenceRcreates}
If in a Quillen pair $L\colon \cC\rightleftarrows \cD\colon R$ the right adjoint $R$ creates weak equivalences and the unit on any cofibrant object is a weak equivalence, then the Quillen pair is a Quillen equivalence.
\end{prop}

\begin{proof}
To show that $(L,R)$ is a Quillen equivalence it is enough to show that for a cofibrant $X$ in $\cC$ and a fibrant $Y$ in $\cD$, a map $f:L(X) \to Y$ is a weak equivalence in $\cD$ if and only if its adjoint $f^{\sharp}: X \to R(Y)$ is a weak equivalence in $\cC$. 

We first observe that the derived unit at $X$ coincides with the unit $\eta_X$, since $R$ preserves all weak equivalences, and it is therefore a weak equivalence. Given that the following diagram commutes 
$$
\xymatrix@C=4pc{X \ar[r]^-{\eta_X}_-{\simeq} \ar[dr]_-{f^{\sharp}}& R(L(X)) \ar[d]^-{R(f)} \\ 
& R(Y),}
$$
the adjoint map $f^{\sharp}$ is a weak equivalence if and only if $R(f)$ is a weak equivalence. Finally, since $R$ creates weak equivalences, this is true if and only if $f$ is a weak equivalence.
\end{proof}

We are now ready to prove the Quillen equivalence.

\begin{proof}[Proof of \cref{Quillenequivalence}]
By \cref{rightinverse2}, the unit of this adjunction is an isomorphism. Since $R$ preserves all weak equivalences the unit on any cofibrant object is also the derived unit. Given that $R$ creates weak equivalences, it follows from \cref{QuillenequivalenceRcreates} that the adjunction is a Quillen equivalence.
\end{proof}

\begin{digression}
The adjunction
$$L\colon\sset\rightleftarrows\pder^{st}\colon R$$
is of the same nature as the adjunction
$$p_1^*\colon \sset\rightleftarrows\ssset\colon i_1^*$$
from \cite[\textsection 4]{JT}, where the right adjoint $i_1^*$ restricts a bisimplicial set to its $0$-th row. This adjunction was also proven in \cite[Theorem 4.11]{JT} to be a Quillen equivalence between the Joyal model structure on $\sset$ and the Rezk model structure for complete Segal spaces on $\ssset$.

However, the role played by the two adjoint pairs of functors is not fully analogous, in that the model structure for complete Segal spaces is not transferred from the Joyal model structure using the right adjoint $i_1^*$. A bisimplicial set whose $0$-th row is a quasi-category is not necessarily a complete Segal space.

Notice that we can use the same technique as the one presented above to transfer the Joyal model structure to the category of bisimplicial sets, using the functor $i^*_1$, obtaining a Quillen equivalence. The identity functor from the Rezk model structure to this transferred model structure is right Quillen.

\end{digression}

We now give  a more explicit description of  the model structure from \cref{transferredmodelstructure} on the category $\pder^{st}$, starting with the class of cofibrations.

Then we will characterize the fibrant prederivators, the fibrations between fibrant prederivators and the acyclic fibrations in terms of suitable lifting properties. The following results are straightforward consequences of \cref{transferredmodelstructure} and the characterisations of corresponding classes of maps in the Joyal model structure on $\sset$, which were recalled in \cref{Joyalmodelstructure}.  

We begin by introducing some notation.

\begin{notn}
\label{horn}
For $n\ge0$ and $0\le k\le n$, the $k$-th horn of the representable prederivator $D_{[n]}$ is the prederivator
$$\boldsymbol{\Lambda}^k_{[n]}:=L(\Lambda^k[n]).$$
It comes with a canonical cofibration of prederivators
$$\boldsymbol{\Lambda}^k_{[n]}\to D_{[n]}$$
induced by the simplicial horn inclusion $\Lambda^k[n]\hookrightarrow\Delta[n]$.
\end{notn}

\begin{notn} For $n\ge1$, the boundary of the representable prederivator $D_{[n]}$ is the prederivator
$$\boldsymbol{\partial}D_{[n]}:=L(\partial\Delta[n]).$$
It comes with a canonical cofibration of prederivators
$$\boldsymbol{\partial}D_{[n]}\to D_{[n]}$$
induced by the simplicial boundary inclusion $\partial\Delta[n]\hookrightarrow\Delta[n]$.
\end{notn}

\begin{cor}
A map of prederivators $\varphi\colon\mathbb D\to\mathbb E$ is a cofibration if and only if it can be expressed as a retract of transfinite compositions of pushouts of boundary  maps $\boldsymbol{\partial}D_{[n]}\to D_{[n]}$ for $n\ge1$.
\end{cor}

\begin{cor}
\label{fibrantlifting}
A prederivator $\mathbb D$ is fibrant if and only if it has the right lifting property  with respect to all inner horn cofibrations  $\boldsymbol{\Lambda}^k_{[n]}\to D_{[n]}$ for $n\ge2$ and $0<k<n$.
\end{cor}

\begin{cor}
A map between fibrant prederivators $\varphi\colon \mathbb D\to\mathbb E$ is a fibration if and only if it has the right lifting property with respect to all inner horn cofibrations  $\boldsymbol{\Lambda}^k_{[n]}\to D_{[n]}$ for $n\ge2$ and $0<k<n$ and with respect to either of the two canonical maps $D_{[0]}\to D_{\mathbb I}$.
\end{cor}

\begin{cor}
A map of prederivators $\varphi\colon\mathbb D\to\mathbb E$ is an acyclic fibration if and only if it has the right lifting property  with respect to all boundary cofibrations   $\boldsymbol{\partial}D_{[n]}\to D_{[n]}$ for $n\ge1$.
\end{cor}

\begin{prop}
Any cofibration of prederivators $\varphi\colon\mathbb D\to\mathbb E$ induces an injective function
$$\Ob(\varphi_J)\colon\Ob(\mathbb D(J))\to\Ob(\mathbb E(J))$$
for any $J\in\catfin$.
\end{prop}

\begin{proof}
The boundary cofibrations $\boldsymbol{\partial}D_{[n]}\to D_{[n]}$ induce injective functions
$$\Ob(\boldsymbol{\partial}D_{[n]}(J))\to\Ob(D_{[n]}(J))$$
for any $J\in\catfin$.
The result now follows from the fact that the functor $\Ob$ preserves retracts, transfinite compositions, and pushouts.
\end{proof}

\begin{rmk}
By Condition (3) of \cref{definitionrepresentable}, all quasi-representable prederivators are fibrant. Moreover, every prederivator is weakly equivalent to a quasi-representable one. Indeed, if $f\colon \mathbb D\to\widetilde{\mathbb D}$ is a fibrant replacement for a prederivator $\mathbb D$, we can consider the zig-zag
$$\mathbb D\stackrel{f}{\longrightarrow}\widetilde{\mathbb D}\stackrel{\varepsilon_{\widetilde{\mathbb D}}}{\longleftarrow} LR(\widetilde{\mathbb D})\stackrel{}{\longrightarrow} \rep{R(\widetilde{\mathbb D})},$$
where the last map is the adjoint of the isomorphism
$$R(\widetilde{\mathbb D})\cong R\rep{R(\widetilde{\mathbb D})}$$
from \cref{rightinverse2}.
They are all weak equivalences, as a consequence of \cref{rightinverse,rightinverse2}.

Thus, the homotopy theory of (fibrant) prederivators is recoved by the homotopy theory of quasi-representable prederivators, which is isomorphic by \cref{theoremcarlson} to the homotopy theory of quasi-categories.
\end{rmk}

Given that $\mathbb Ho$ is fully faithful from \cref{theoremcarlson}, the following proposition gives a complete description of weak equivalences between prederivators that are in the image of the functor $\mathbb Ho\colon\sset\to\pder^{st}$.

\begin{prop}
The functor $\mathbb Ho$ creates weak equivalences, i.e., a map of simplicial sets $f\colon X\to Y$ is a categorical equivalence if and only if the induced map of prederivators $\rep{f}\colon\rep{X}\to\rep{Y}$ is a weak equivalence.
\end{prop}

\begin{proof}
Given the isomorphism $R\mathbb Ho\cong\id_{\sset}$ from \cref{rightinverse}, a map of simplicial sets $f\colon X\to Y$ is a categorical equivalence if and only if
$$R\rep{f}\colon R\rep{X}\to R\rep{Y}$$
is also one. Given that weak equivalences of prederivators are by definition created by the functor $R$, this is equivalent to saying that $\rep{f}\colon\rep{X}\to\rep{Y}$ is a weak equivalence of prederivators.
\end{proof}

\begin{prop}
If $f\colon X\to Y$ is a categorical equivalence between quasi-categories, the induced map $\rep f\colon\rep{X}\to\rep{Y}$ is levelwise an equivalence of categories.
\end{prop}
\begin{proof}
Suppose that $f\colon X\to Y$ is a categorical equivalence between quasi-categories.
For any category $J$, the induced map $f^{NJ}\colon X^{NJ}\to Y^{NJ}$ is a categorical equivalence, since the Joyal model structure is enriched over itself. By \cite[\textsection1.12]{Joyalnotes} it induces an equivalence of homotopy categories $\ho(f^{NJ})\colon \ho(X^{NJ})\to \ho(Y^{NJ})$, and this is by definition precisely $$\rep{f}(J)\colon\rep{X}(J)\to\rep{Y}(J),$$ as desired.
\end{proof}

\begin{rmk}
Notice that $R$ preserves generating cofibrations since $RL\cong \id_{\sset}$ and it preserves all weak equivalences by definition. By \cref{adjointsofR}, the functor $R$ has a right adjoint $U$ and thus
$$R\colon\pder^{st}\rightleftarrows\sset\colon U$$
is a Quillen pair, which is in fact a Quillen equivalence.
\end{rmk}

We conclude with a brief discussion on further directions.
\begin{digression}
\label{finaldigression}
With a model structure on $\pder^{st}$ that is Quillen equivalent to the Joyal model structure on $\sset$ the next natural goal is to attempt to endow the category of fibrant objects of $\pder^{st}$ with the structure of an $\infty$-cosmos, in the sense of \cite[Definition 2.1.1]{RV4}, and to attempt to show that this $\infty$-cosmos is biequivalent to the $\infty$-cosmos $\qcat$ of quasicategories, defined in \cite[Example 2.1.4]{RV4}. Having equivalent $\infty$-cosmoi guarantees that the $2$-category theory in each $\infty$-cosmos is the same, so in particular the notions of limits and colimits, adjunctions, and cartesian fibrations coincide.

Mimicking what is done in similar scenarios (e.g. when starting with the model structure for complete Segal spaces, Segal categories, and $1$-complicial sets, see \cite[\textsection2]{RV4}), one would attempt to upgrade the model structure on $\pder^{st}$ to one enriched over the Joyal model structure on $\sset$, such that the functors $L$ and $R$ form a simplicial Quillen adjunction.
The standard technique to do this is to use \cite[Proposition 2.2.3]{RV4} to transfer the simplicial enrichment, which works whenever the left adjoint $L$ is strong symmetric monoidal.

However, by \cref{Ldoesnotrespectproducts}, the functor $L$ does not preserve products. By \cite[Proposition 3.7.10]{RiehlCHT},
 one can deduce that the functor $R$ is then not compatible with simplicial cotensors.  These two equivalent conditions
obstruct the simplicial enrichment on $\pder^{st}$ discussed in \cref{pDersimplicial} from giving an enrichment over the Joyal model structure, and indeed prevent the simplicial lifts of $L$ and $R$ from even being simplicially adjoint functors.  This means that the conditions of  \cite[Proposition 2.2.3]{RV4} are not satisfied.

At this moment, the authors have been unable to provide another method to achieve the desired $\infty$-comos structure on the category of fibrant prederivators using the functor $R$. Nonetheless,
as a consequence of Carlson's Theorem (which was stated as \cref{theoremcarlson}) the simplicial category of quasi-representable prederivators is isomorphic via the functor $\mathbb Ho$ to the simplicial category of quasicategories, which is indeed an $\infty$-cosmos. This in itself allows for several interesting results, and is the subject of a future project.
\end{digression}

\bibliographystyle{alpha}
\bibliography{ref}

\end{document}